\documentclass[11pt]{amsart}
\usepackage{amscd,amssymb}
\usepackage{graphicx}
\usepackage[all]{xy}

\theoremstyle{plain}

\newtheorem{theorem}{Theorem}[section]
\newtheorem{thm}[equation]{Theorem}
\newtheorem{prop}[equation]{Proposition}
\newtheorem{cor}[equation]{Corollary}

\newtheorem{lemma}[equation]{Lemma}

\numberwithin{equation}{section}

\newcommand{\R}{\mathbb R}

\newcommand{\A}{\mathbb A}

\def\Hom{{\rm Hom}}

\def\Aut{{\rm Aut}}

\def\vol{{\rm vol}}

\def\SL{{\rm SL}}

\def\GSp{{\rm GSp}}

\def\PGSp{{\rm PGSp}}

\def\Sp{{\rm Sp}}
\def\Spin{{\rm Spin}}

\def\SU{{\rm SU}}

\def\U{{\rm U}}

\def\GL{{\rm GL}}
\def\PGL{{\rm PGL}}

\def\SO{{\rm SO}}
\def\Ind{{\rm Ind}}
\def\Sp{{\rm Sp}}

\def\tr{{\rm tr\,}}

\def\A{{\mathbb A}}

\def\R{{\mathbb R}}

\def\n{{\bf n}}

\title {An exceptional Siegel-Weil formula and \\ Poles of  the Spin L-function of $\PGSp_6$}

\author{Wee Teck Gan and Gordan Savin}

\address{W.T.G.:   Department of Mathematics, National University of Singapore, 10 Lower Kent Ridge Road
Singapore 119076} \email{matgwt@nus.edu.sg}
\address{G. S.: Department of Mathematics, University of Utah, Salt Lake City, UT 84112, USA}\email{savin@math.utah.edu}
\begin{document}
\maketitle

\begin{abstract}
 We show a Siegel-Weil formula in the setting of exceptional theta correspondence. Using this, together with a new Rankin-Selberg integral for the Spin L-function of $\PGSp_6$ discovered by A. Pollack, we prove that a cuspidal representation of $\PGSp_6$ is a (weak) functorial lift from the exceptional group $G_2$ if  its (partial) Spin L-function has a pole at $s=1$.
 \end{abstract}

\section{Introduction}

Let $F$ be a totally real number field, and $\mathbb A$ its ring of ad\`eles. 
Let $\pi \cong \otimes_v \pi_v$ be an irreducible cuspidal automorphic representation of 
the group $\PGSp_6(\mathbb A)$, which is unramified outside a finite set  $S$  of places (including all real places). 
 Since the Langlands dual group of $\PGSp_6$ is $\Spin_7(\mathbb C)$, there is an associated semi-simple conjugacy class $s_v$ in $\Spin_7(\mathbb C)$ for $v \notin S$;
 this is the Satake parameter of the local component $\pi_v$. 
  If  $r$ denotes  the 8-dimensional spin representation of $\Spin_7(\mathbb C)$, the partial spin $L$-function corresponding to $\pi$ is defined to 
 be the product 
 \[ 
 L^S(s,\pi, {\rm Spin}) = \prod_{v\notin S} \frac{1}{\det( 1-  r(s_v) q_v^{-s})}
 \]
 where $q_v$ is the order of the residual field of the local field $F_v$. 
 \vskip 5pt
 
 It is well known that the stabilizer in $\Spin_7(\mathbb C)$ of a generic vector in the spin representation 
 is the exceptional group $G_2(\mathbb C)$, giving a well-defined conjugacy class of embedding 
 \[ \iota: G_2(\mathbb C) \longrightarrow \Spin_7(\mathbb C). \]
   Therefore, as a special case of the Langlands functoriality principle, if
 $L_S(s,\pi, {\rm Spin})$ has a simple pole at $s=1$,  then one expects $\pi$ to be a functorial lift from an exceptional group of absolute type $\mathbf G_2$ defined over $F$.
 We note that every such group  is given as the  automorphism group of an octonion algebra $\mathbb O$ over $F$, and by the Hasse principle, the 
 number of isomorphism classes of such groups is $2^n$ where $n$ is the number of real places of $F$. 
 \vskip 5pt
 
  As explained in a recent paper of Chenevier \cite[\S 6.12]{C}, if $\pi$ is a tempered cuspidal representation of $\PGSp_6$ such that for almost all places $v$, the Satake parameter $s_v$ of $\pi_v$   belongs to $\iota(G_2(\mathbb C))$ (or more accurately,  the conjugacy class $s_v$ meets $\iota(G_2(\mathbb C))$), then $L^S(s, \pi, {\rm Spin})$ will have a pole at $s=1$ and so one expects such a tempered $\pi$ to be a functorial lift from $G_2$. 
  In this paper we also prove a slightly weaker version of this expectation:
  \vskip 5pt
  
  \begin{thm}  \label{T:main}
  In the above setting, with $F=\mathbb Q$, suppose that $\pi$ is a cuspidal automorphic representation of $\PGSp_6$ such that $L^S(s, \pi, {\rm Spin})$ has a pole at $s=1$. Then  there exists an octonion algebra $\mathbb O$ over $F$ and  a cuspidal  automorphic representation $\pi'$ of $\Aut(\mathbb O)$ such that the Satake parameters of $\pi'$ 
  are mapped by $\iota$ to those of $\pi$ (i.e. $\pi$ is  a weak functorial lift of $\pi'$).   
  \vskip 5pt
  
  If the cuspidal representation $\pi$ of $\PGSp_6$ is tempered, then the following are equivalent:
  \vskip 5pt
  \begin{itemize}
  \item[(a)]  For almost all places $v$, the Satake parameter $s_v$ of $\pi_v$ is contained in $\iota(G_2(\mathbb C))$.
  
  \item[(b)] There exists an octonion algebra $\mathbb O$ over $F$ and  a cuspidal automorphic representation $\pi'$ of $\Aut(\mathbb O)$ such that  $\pi$ is a weak functorial lift of $\pi'$.
    \end{itemize}
      \end{thm}
  
  Since the local Langlands classification is not known for $\mathbf G_2$ and for $\PGSp_6$, this is essentially the best possible result one can expect at the moment. However, 
  if $\pi$ is unramified everywhere or if it corresponds to a classical Siegel modular form, then $\pi$ is a functorial lift. 
  Special cases of this result were previously obtained by Ginzburg and Jiang \cite{GJ}, Gan and Gurevich \cite{GG} and Pollack and Shah \cite{PS}. 
  \medskip 
  
 Our proof of Theorem \ref{T:main} is based on the following three ingredients:
 \vskip 5pt
 
 \begin{itemize}
 \item[(1)] An exceptional theta correspondence for the dual pair $\Aut(\mathbb O) \times \PGSp_6$ arising from 
 the minimal representation $\Pi$ of a group of absolute type $\mathbf E_7$. 
 
 \vskip 5pt
 
 \item[(2)]  A Siegel-Weil formula proved in this paper; see Theorem \ref{T:sw} below. 
 \vskip 5pt
 
 \item[(3)]   An integral representation of the spin $L$-function of $\pi$ recently discovered by A. Pollack \cite{P}. This work is over $F=\mathbb Q$ and is the source of the same 
 restriction in Theorem \ref{T:main}. 
 \end{itemize}
 \vskip 5pt
 
 In greater detail, let $J$ be the exceptional Jordan algebra of $3\times 3$ hermitian symmetric matrices with coefficients in an octonion algebra $\mathbb O$. By the Koecher-Tits 
 construction, the algebra $J$ gives rise to an adjoint group $G$ of absolute type $\mathbf E_7$, with a maximal parabolic subgroup $P=MN$, such that the unipotent 
 radical $N$ is commutative and isomorphic to $J$. Since $G$ is adjoint, the conjugation action of $M$ on $N$ is faithful, and $M$ is isomorphic to the similitude group 
  of the natural cubic norm form on $J$. Thus the natural action of $\Aut(\mathbb O)$ on $J$ gives an embedding of $\Aut(\mathbb O)$  into $M$. The centralizer 
 of $\Aut(\mathbb O)$ is $\PGSp_6$. 
 To see this, observe that the centralizer of $\Aut(\mathbb O)$ in $J$ is the Jordan subalgebra $J_F$  of $3\times 3$  symmetric matrices with coefficients in $F$. The group 
 $\PGSp_6$ arises from $J_F$ by the Koecher-Tits construction. This gives the dual pair 
 \[  \Aut(\mathbb O) \times \PGSp_6 \subset G \]
  alluded to in (1) above. 
 \vskip 5pt

 We can now describe another dual pair in $G$.
 Let $D$ be a quaternion algebra over $F$, and assume that we have an embedding 
 $i: D \rightarrow \mathbb O$. The centralizer of $D$ in $\Aut(\mathbb O)$ is isomorphic to $D^1$, the group of norm one elements in $D$. Conversely, the centralizer 
 (i.e. the point-wise stabilizer) of $D^1$ in $\mathbb O$ is $i(D) \subset \mathbb O$. Thus the centralizer of $D^1$ in $J$ is the Jordan subalgebra $J_D$ 
 of $3\times 3$ hermitian symmetric matrices with coefficients in $D$, and the centralizer of $D^1$ in $G$ is a group $G_D$ of absolute type $\mathbf D_6$ 
 arising from $J_D$ by the Koecher-Tits construction.  Thus we have a dual pair
 \[  D^1 \times G_D \subset G. \]
Indeed, the two dual pairs we have described fit into the following see-saw diagram, where the vertical lines 
 represent inclusions of groups: 
\[
 \xymatrix{
  \Aut(\mathbb O)  \ar@{-}[dr] \ar@{-}[d] & G_D
     \ar@{-}[d] \\
  D^1 \ar@{-}[ur] &  \PGSp_6}
\] 
 
 \vskip 5pt
 The Siegel-Weil formula mentioned in (2) above concerns the global theta lift 
$\Theta(1)$  of the trivial representation of $D^1$ to $G_D$, obtained by restricting the minimal representation $\Pi$ of $G$ to the dual pair $D^1\times G_D$. 
Roughly speaking, $\Theta(1)$ is the space of automorphic functions on $G_D$ obtained by averaging the functions in $\Pi$ over $D^1(F)\backslash D^1(\mathbb A)$. 
 We prove that $\Theta(1)$ is an 
irreducible automorphic representation of $G_D$ and determine its local components (as abstract representations) by computing the corresponding local theta lifts. 
We have not computed the local theta lift for complex groups, and this is the source of the restriction in the paper to totally real fields $F$. The Siegel-Weil formula identifies the functions in $\Theta(1)$ as residues of certain Siegel-Eisenstein series.
\vskip 5pt

More precisely, since $G_D$ arises from $J_D$ by the Koecher-Tits construction, 
it contains a maximal parabolic subgroup with abelian unipotent radical isomorphic to $J_D$. Let $E_D(s,f)$ be the degenerate Eisenstein series attached to this maximal parabolic 
subgroup, where $s\in \mathbb R$ and $f$ values over all standard sections of the corresponding degenerate principal series representation $I_D(s)$. In \cite{HS}, it was proved that $E_D(s,f)$ has at most a simple pole at $s=1$, and the residual representation
\[   \mathcal E_D  := \{ \mathrm{Res}_{s=1} E_D(s, f) : f \in I_D(s) \}  \]
was completely determined. Our main result is the following Siegel-Weil identity in the space of automorphic forms of $G_D$: 
\begin{thm}  \label{T:sw}
For fixed quaternion $F$-algebra $D$,  we have:
\[ 
 \mathcal E_D  = \oplus _{i: D \rightarrow \mathbb O} \Theta(1). 
\] 
Here the sum is taken over all isomorphism classes of embeddings $i: D \rightarrow \mathbb O$ into octonion algebras over $F$. 
\end{thm}
We emphasize that $D$ is fixed here 
but $\mathbb O$ vary. If $D$ is split, i.e. a matrix algebra, then $\mathbb O$ is also split, and there is only one term on the right. In general the number of summands 
on the right is equal to $2^m$ where $m$ is the number of real places $v$ of $F$ such that $D_v$ is a division algebra. 
\vskip 10pt

At this point, we need the result of A. Pollack \cite{P}: there exists a quaternion algebra $D$ such that the partial spin-$L$-function $L_S(\pi,s)$ is given 
as an integral, over $\PGSp_6$,   of a function $h \in \pi$ against the Eisenstein series $E_D(s,f)$. Thus, if the $L$-function has a pole at $s=1$, then the integral of $h$ 
against the elements of $\mathcal E_D$ is non-zero. The Siegel-Weil identity (i.e. Theorem \ref{T:sw}) then implies that $\pi$ appears in the exceptional theta correspondence
for the dual pair $\Aut(\mathbb O) \times \PGSp_6$, for some $\mathbb O$ containing $D$. 
Since this exceptional theta correspondence  is known to be 
 functorial for spherical representations (see \cite{LS} and \cite{SW15}), this completes the proof that $\pi$ is a weak lift from a group of absolute type $\mathbf G_2$. 
\vskip 15pt

\section{Groups}

\subsection{\bf Octonion algebra.} 
Let $F$ be a field of characteristic 0, and $D$ be a quaternion algebra over $F$. It is a $4$-dimensional associative and non-commutative algebra over $F$ which 
  comes equipped with a conjugation map $x \mapsto \overline{x}$ with associated norm $N(x)  = x  \overline{x} = \overline{x} x$ and trace $\tr(x)  = x + \overline{x}$. Moreover, $N : \mathbb{O}  \rightarrow F$ is a nondegenerate quadratic form.   
 \vskip 5pt
 
 An octonion algebra $\mathbb O$ over $F$ is obtained by doubling the quaternion algebra $D$. More precisely, 
fix a non-zero element $\lambda$ in $F$. As a vector space over $F$, 
$\mathbb O$ is a set of pairs $(a,b)$ of elements in $D$. The multiplication is defined by the formula 
\[ 
(a,b)\cdot (c,d) = (ac+ \lambda d\bar b, \bar a d+ cb) . 
\] 
If $x=(a,b)$, then the conjugation map is  $\bar{x}=(\bar a, -b)$, so that  $N(x)= x\cdot \bar x= N(a) - \lambda N(b)$ is the norm and $\tr(x)= x+ \bar x= \tr(a)$ the trace on  $\mathbb O$.
 In particular, $\mathbb O$ is split if $\lambda$ is a norm of an element in $D$. Every element $x$ of $\mathbb{O}$ satisfies its characteristic polynomial $t^2 - \tr(x) t + N(x)$.  
 The automorphism group $\Aut(\mathbb O)$ of the $F$-algebra $\mathbb{O}$ is an exceptional group of the Lie type $\mathbf G_2$. 
  It is a simple linear algebraic group of rank $2$ which is both simply connected and adjoint.  The algebra $D$ is naturally a subalgebra of $\mathbb O$, consisting of 
  all $x=(a,0)$. Let $D^1$ be the group of norm one elements in $D$. Then any $g\in D^1$ acts as an automorphism 
of $\mathbb O$ by $g\cdot (a,b) = (a, b\bar g)$ for all $(a,b) \in \mathbb O$.  The subgroup $D^1 \subset \Aut(\mathbb O)$ is precisely the point-wise stabilizer of the subalgebra $D \subset \mathbb O$.

\vskip 5pt

\subsection{Albert algebra} An Albert algebra is an exceptional $27$-dimensional Jordan algebra $J$ over $F$.  It can be realized as the set of matrices 
\[ 
A=\left(\begin{array}{ccc} 
\alpha & x & \bar z \\
\bar x & \beta & y \\
z & \bar y & \gamma
\end{array}\right)
\] 
where $\alpha,\beta,\gamma\in F$ and $x,y,z\in \mathbb O$. The determinant $A\mapsto \det A$ defines a natural cubic form on $J$. Let $M$ be the similitude group of this cubic form.   It is a reductive group of semisimple type $E_6$.  The $M$-orbits in $J$ are classified by the rank of the matrix $A$. Without going into a general definition of 
the rank, we say that $A\neq 0$ has rank one, if $A^2= \tr(A) \cdot A$. Explicitly, this means that the entries of $A$ satisfy the equalities 
\[ 
N(x)= \alpha\beta, \, N(y)=\beta\gamma, \, N(z)=\gamma\alpha, \, \gamma\bar x=yz, \, \alpha \bar y= zx, \, \beta \bar z = x y. 
\]

\subsection{Dual pairs}  \label{SS:dual pair} 
Assume that $G$  is a reductive group over $F$, adjoint and of absolute type $E_7$, arising from the Albert algebra $J$ via the Koecher-Tits construction.   
For our purposes it will be more convenient to realize $G$ as a quotient, modulo one dimensional center $C\cong F^{\times}$, of a reductive group $\tilde G$ 
acting on the 56-dimensional representation $W=F + J + J +F$.  In particular, $G$ acts on the projective space $\mathbb P(W)$. Let $P$ be a maximal parabolic 
and  $\bar P$ its opposite, defined as fixing the points $(1,0,0,0)$ and $(0,0,0,1)$ in $\mathbb P(W)$. Then $P=MN$ where $N$ is the unipotent radical and 
 $M=P\cap \bar P$ a Levi. Then $M$ is isomorphic to 
the similitude group  of the cubic form $\det$ on $J$, and $N\cong J$, as $M$-modules. 
\vskip 5pt

Recall that we have constructed $\mathbb O$ by doubling a quaternions subalgebra $D$. Let $J_F$ and $J_D$ be the subalgebras consisting of all elements in 
$J$  with off-diagonal entries in $D$ and $F$, respectively.  Let $J_0=F$ be the scalar subalgebra of $J$. Consider a sequence of simple, simply connected groups 
\[ 
 D^1 \subset \Aut(\mathbb O) \subset \Aut(J) 
 \]  
 where an element in $\Aut(\mathbb O)$ acts on the off-diagonal entries of elements in $J$. The point-wise stabilizers in $J$ of these three groups are, respectively, 
 \[ 
 J_D \supset J_F \supset J_0=F. 
 \] 
 Observe that $\Aut(J)$ naturally acts on $W$, giving 
an embedding $\Aut(J) \subset \tilde G$. The centralizers in $\tilde G$ of the three groups in the sequence are, respectively, 
\[ 
\tilde G_D \supset \GSp_6(F) \supset \GL_2(F) 
\] 
These three groups act on $32$, $14$ and $4$-dimensional subspaces of $W$ obtained by replacing $J$ by $J_D$, $J_F$ and $J_0$, respectively. 
 It is worth mentioning that the 4-dimensional representation of $\GL_2(F)$ is the symmetric cube of the standard 2-dimensional representation, twisted by $\det^{-1}$.  
 The group $\tilde G_D$ acts on 
 \[ 
 W_D = F + J_D + J_D + F. 
 \] 
A detailed description of $\tilde G_D$ and the action on $W_D$ in in the Pollack's paper \cite{P}.  Let $G_D$ be the quotient of $\tilde G_D$ by its center $C\cong F^{\times}$. 
Then $D^1 \times G_D$ is a dual pair in $G$, mentioned in the introduction. 
\vskip 5pt

Let $P_D=M_DN_D=G_D\cap P$. With the identification $N\cong J$ fixed, we have $N_D\cong J_D$. 
  The group $P_D$ is a maximal parabolic subgroup of type $\mathbf A_5$. 

\vskip 10pt

\section{Minimal representation}

Let $F$ be a real or $p$-adic field. Let $I(s)$ be the degenerate principal series representation of $G$ attached to $P$ where $s\in \mathbb R$. We normalize $s$ as in 
\cite{We} so that the trivial representation is a quotient and a submodule at $s=9$ and $s=-9$ respectively, whereas the minimal representation $\Pi$ is a quotient and a sub-module at $s=5$ and  $s=-5$, respectively.  

\subsection{Unitary model} 
Fix $\psi : F \rightarrow \mathbb C^{\times}$, a non-trivial additive character, unitary if $F=\mathbb R$. 
 After identifying $N\cong J$,   any $A\in J$ defines a character of $N$ given by 
 \[ 
\psi_A(B) = \psi (\tr(A\circ B)) = \psi_B(A) 
\] 
for $B\in J$, where $A\circ B$ denotes the Jordan multiplication. Every unitary character of $N$ is equal to $\psi_A$ for some $A$.  
Let $\Omega \subseteq J$ be the set of rank one elements in $J$. 
A unitary model of the minimal representation is $\mathcal H= L^2(\Omega)$. Here only the acton of the maximal parabolic $P=MN$ is obvious: the group 
$M$ acts geometrically, while $A\in J\cong N$ acts by multiplying  by $\psi_A$. 
\vskip 5pt

\subsection{Smooth model}
We have the following \cite{KS15}. 

\begin{thm} \label{T:model} 
 Let $\Pi$ be the subspace of $G$-smooth vectors in the unitary minimal representation $\mathcal H$. Then 
 \[  C_c^{\infty} (\Omega) \subset \Pi \subset C^{\infty}(\Omega). \]
 If $F$ is $p$-adic, then 
 \[  \Pi_N \cong \Pi/  C_c^{\infty}(\Omega) \quad \text{as $M$-modules.} \]
  If $A\in J$ is nonzero,    then any continuous functional $\ell$ on $\Pi$ such that $\ell( B \cdot f) = \psi_A(B) \cdot \ell(f)$ for all $B\in N$ and $f\in \Pi$ is equal to a multiple of  the evaluation map $\delta_A(f)=f(A)$.  In particular, $\ell=0$ if $A$ is not of rank one. 
 
\end{thm} 

 \vskip 5pt
 
 \subsection{Spherical vector} 
It is not so easy to characterise the subspace $\Pi \subset C^{\infty}(\Omega)$. However, 
we can describe a spherical vector in $\Pi$ in the split case. The algebra $\mathbb O$ is obtained by doubling the matrix algebra $D=M_2(F)$ with $\lambda=1$. 
Assume firstly that $F$ is a $p$-adic field. 
Let $O$ be the ring of integers in $F$ and $\varpi$ a uniformizing element.  
 We have an obvious integral structure on $D$ (the lattice of integral matrices), and hence on $\mathbb O$, the integral lattice being the set of pairs 
 $(a,b)$ where $a,b\in M_2(O)$. This lattice is a maximal order in $\mathbb O$.  Now we have an integral structure on $J$ so that $J(O)$ is 
the set of elements $A\in J$ such that the diagonal entries are integral, and off diagonal contained in the maximal order in $\mathbb O$. 
 The greatest common divisor of entries of $A\in J(O)$, is simply the largest power $\varpi^n$ dividing $A$ i.e. such that $A/\varpi^n$ is in $J(O)$.  
 We have the following \cite{SW07}:

\begin{thm} \label{T:WS}
Assume $G$ is split and $F$ a $p$-adic field. 
 Assume the conductor of $\psi$ is $O$.  Then the spherical vector in $\Pi$ is a function $f^{\circ}\in C^{\infty}(\Omega)$ supported in $J(O)$. 
 Its value at $A\in \Omega$ depends on the gcd of entires of $A$. More precisely, if the gcd of $A$ is $\varpi^n$, and $q$ is the order of the residual field, then 
 \[ 
 f^{\circ} (A)= 1+ q^3 + \ldots + q^{3(n-1)}. 
 \] 
 \end{thm} 
 \vskip 5pt
Since $\Pi$ is generated by $f^{\circ}$ as a $P$-module, and the action of $P$ on $\Pi$ is easy to describe, this theorem gives us a good handle on $\Pi$.
 \vskip 5pt
 
  Assume now that $F=\mathbb R$; in this case, one has a similar result due to Dvorsky-Sahi \cite{DS99}.
For every $a\in M_2(\mathbb R)$, let  $||a||^2$ is the sum of squares of its entries. For $x=(a,b)\in \mathbb O$, let 
$||x||^2=||a||^2 + ||b||^2$. Extend this to $A\in J$ by 
\[ 
||A||^2 = \alpha^2 + \beta^2 + \gamma^2 + ||x||^2 + ||y||^2 + ||z||^2. 
\] 
Let $K_{3/2}(u)$ denote the modified Bessel function of the second kind. Recall that $K_{3/2}(u) >0$, for $u>0$,  and is rapidly decreasing as $u\rightarrow +\infty$. 
Then  \cite[Theorem 0.1]{DS99}: 

\begin{thm}  \label{T:DS}
Assume $G$ is split and $F=\mathbb R$. 
 Then the spherical vector in $\Pi$ is a function $f^{\circ} \in C^{\infty}(\Omega)$ given by 
  \[ 
 f^{\circ} (A)= ||A||^{-3/2} K_{3/2} (|| A||). 
 \] 
 \end{thm}

 \vskip 15pt
 \section{Local Theta Lifts: $p$-adic case}
 In this section, let $F$ be a $p$-adic field, so that the octonion algebra $\mathbb O$ is split. We are interested in understanding the theta lift of the trivial representation of $D^1$ to the group $G_D$.
 \vskip 5pt

 \subsection{$N_D$-spectrum} A crucial step is to understand the $N_D$-spectrum of the minimal representation $\Pi$.  In this case we have an exact sequence of $P$-modules 
 \[ 
0\rightarrow C_c^{\infty}(\Omega) \rightarrow \Pi  \rightarrow \Pi_N \rightarrow 0. 
\] 
The  characters of $N_D\cong J_D$ are identified with the elements in $J_D$ using the trace paring, as we did for $J$. We shall only need three characters, denoted by 
$\psi_1$, $\psi_2$ and $\psi_3$, corresponding to the elements 
\[ 
\left(\begin{array}{ccc} 
\pm 1 & 0 & 0 \\
0 & 0 & 0 \\
0 & 0 & 0 
\end{array}\right),  \text{ } 
\left(\begin{array}{ccc} 
\pm 1 & 0 & 0 \\
0 & \pm 1 & 0 \\
0 & 0 & 0 
\end{array}\right) \text{ and } 
\left(\begin{array}{ccc} 
\pm 1 & 0 & 0 \\
0 & \pm 1 & 0 \\
0 & 0 & \pm 1 
\end{array}\right)
\] 
of rank 1, 2 and 3, respectively. We need to allow signs to capture all possible rank 1, 2 and 3 orbits in the real case. 
 The following lemma is one of the keys in this paper, and we emphasize that we \underline{do not} assume that $D$ is split here. 

\begin{lemma} \label{L:key} 
Let $\Pi$ be the minimal representation of $G$. Then: 
\begin{itemize} 
\item[(i)]  $\Pi_{N_D,\psi_3}=0$. 
\item[(ii)]  $\Pi_{N_D,\psi_2}\cong C_c^{\infty}(D^1)$, as $D^1$-modules. 
\item[(iii)]  If $D$ is a division algebra, then $\Pi_{N_D,\psi_2}\cong \mathbb C$, as $D^1$-modules. 
\end{itemize} 
\end{lemma} 
\begin{proof} 
Let $\omega_i\subseteq \Omega$ be the set of all $A\in \Omega$ such that the restriction of $\psi_A$ to $N_D$ is equal to $\psi_i$. Then 
\[ 
\Pi_{N_D, \psi_i} \cong C_c^{\infty}(\omega_i). 
\] 
It remains to determine each $\omega_i$. Let's start with $i=3$. Then $\omega_3$ consists of all $A\in \Omega$ such that 
\[ 
A=\left(\begin{array}{ccc} 
\pm 1 & x & -z \\
-x & \pm 1 & y \\
z & - y & \pm 1 
\end{array}\right)
\] 
where $x=(0,a), y=(0,b)$ and $z=(0,c)$ for some $a,b,c\in D$. Since $A\in \Omega$, we further have $A^2=\tr(A) A$. Looking at the off-diagonal terms, we get 
the equations 
\[  \text{$yx=\pm z$, $zy=\pm x$ and $xz=\pm y$.} \]
 But the products $yx$, $zy$ and $xz$ have the second coordinate equal to 0. Hence $z=x=y=0$. But then 
$A$ cannot be a rank 1 matrix. Hence $\omega_3$ is empty, and this proves (i). 
\vskip 5pt

For (ii) we see analogously that $y=z=0$. Now $A$ 
has the rank 1 if and only if the first $2\times 2$ minor is 0. This gives $x^2=\pm 1$. Writing this out, with $x=(0,a)$ we see that $\lambda a\bar a=\pm 1$.  
Hence $\omega_2$ is identified with the set of all elements in $D$ with a fixed non-zero norm. This is a principal homogeneous space for $D^1$. This establishes (ii). 
In the last case it is easy to see that $x=y=z=0$.  
\end{proof} 

We now derive a consequence. Let $\Theta(1)$ be the maximal quotient of $\Pi$ on which $D^1$ acts trivally; it is naturally a $G_D$-module. Lemma \ref{L:key} 
implies that 
\[  \text{$\Theta(1)_{N_D,\psi_3}=0$ and $\Theta(1)_{N_D,\psi_2}=\mathbb C$. } \]
Let $I_D(s)$ be the degenerate principal series representation of attached to 
$P_D$ normalized as in \cite{We}. In particular, the trivial representation is a quotient for $s=5$ and a submodule for $s=-5$.  The inclusion 
$\Pi \rightarrow I(-5)$ composed by the restriction of functions from $G$ to $G_D$ gives a non-zero $D^1$-invariant map $\Pi \rightarrow I_D(-1)$, which clearly factors through 
$\Theta(1)$. By \cite{We} and \cite{HS}, $I_D(-1)$ has a composition  series of length 2. The unique irreducible submodule $\Sigma$ has $N_D$-rank 2. We have:
\vskip 5pt

\begin{cor}  \label{C:theta1}
The above construction gives a surjective $G_D$-equivariant map
\[  \Theta(1)\rightarrow \Sigma \subset I_D(-1) \]
whose kernel has $N_D$-rank no worse than one. If $D$ is a division algebra, then $\Theta(1)\cong \Sigma$. 
\end{cor} 
\begin{proof} 
It remains to prove the last statement. The spherical, rank 2 representation $\Sigma$ is the classical theta lift of the trivial representation of the quaternionic form 
$\SO(4)$. Using the theta correspondence it is easy to check that $\Sigma_{N_D,\psi_1}\cong \mathbb C$.  Thus, from Lemma \ref{L:key} (iii) it follows that the
kernel of the map $\Theta(1)\rightarrow \Sigma$ has $N_D$-rank 0, i.e. $N_D$ acts trivially. 
Since $D^1$ is compact, $\Theta(1)$ is a summand of the minimal representation. By the classical result of Howe-Moore the minimal representation cannot contain 
non-zero vectors fixed by $N_D$. Thus the kernel is trivial. 
\end{proof}

 \vskip 10pt
 
\subsection{Local lifts for split $D$}  We strengthen here the result of Corollary \ref{C:theta1}: $\Theta(1)\cong \Sigma$ even when $D$ is split.  
If $D$ is split then $G$ is split. 
 Let $T \subset G$ be a maximal split torus, so we have the associated root groups. Furthermore, 
$D^1\cong \SL_2$ and it is conjugated to a root $\SL_2$.  
Without loss of generality we can assume that $\SL_2$ corresponds to the 
highest root for some choice of positive roots. Let $T_1=\SL_2\cap T$. Then the centralizer of $T_1$ in $G$ is a Levi subgroup of type $\mathbf D_6$. The Levi $L$ is contained in 
two maximal parabolic subgroups: $Q=LU$ and its opposite $\bar Q= L \bar U$. The unipotent radical $U$ is a two-step unipotent group with the center $U_1$ given by 
the root group corresponding to the highest root. Similarly, the center of $\bar U$ is the root subgroup $\bar U_1$  corresponding to the lowest root. These two root 
groups generate $\SL_2$. The conjugation action of $L$ on $U_1$ and $\bar U_1$ is given by a 
character and its inverse. Hence $G_D$ is the kernel of this character. Since $G$ is of the adjoint type, $G_D$ acts faithfully on 
$U/U_1$, a $32$-dimensional spin representation.  (This representation is not $W_D$, the 32-dimensional representation of $\tilde G_D$, 
from Subsection \ref{SS:dual pair}.) 
We identify $T_1\cong \GL_1$  so that $x\in \GL_1$ acts on 
$U/U_1$ as multiplication by $x$. 

\vskip 5pt 
We now need a result on the restriction of $\Pi$ to the maximal parabolic subgroup $Q=LU$.  By  \cite[Theorem 6.1]{MS97}, the space of $U_1$-coinvariants  of $\Pi$, an $L$-module, sits in an exact sequence 
\[ 
0\rightarrow C_c^{\infty}(\omega) \rightarrow \Pi_{U_1} \rightarrow \Pi_U \rightarrow 0
\] 
where $\omega$ is the $L$-orbit of highest weight vectors in $\bar U /\bar U_1$. The action of $L$ on $C_c^{\infty}(\omega)$ arises from the natural action of 
$L$ on $\omega$ twisted by an unramified character.  

\vskip 5pt

 Let $Q_D=L_DU_D$ be a maximal parabolic subgroup in $G_D$ stabilizing the line through a point $v\in \omega$.  
 Note that the Levi factor $L_D$ of $Q_D$ is also of type $\mathbf A_5$ (like that of $P_D$).
 The action of $Q_D$ on the line gives a homomorphism 
$\chi: Q_D \rightarrow \GL_1$. Thus the stabilizer in $G_D\times \GL_1$ of $v$ consists of all pairs $(g,x)$ such that $g\in Q_D$ and $\chi(g)=x$. 
Since $G_D\times \GL_1$ acts transitively on $\omega$,  it is easy to see that the following holds: 

\begin{thm} \label{T:filtration} 
The normalized Jacquet functor $\pi_{U_1}$, as a $G_D\times \GL_1$-module, has a 2-step filtration with the following quotient and submodule respectively: 
\begin{itemize} 
\item $\Pi_U = \Pi(G_D)\otimes |\cdot|^3 \oplus |\cdot|^5$ where $\Pi(G_D)$ is the minimal representation of $G_D$, and $|\cdot|$ is the absolute value character 
of $\GL_1$. 
\item $\Ind_{Q_D}^{G_D} C_c^{\infty}(\GL_1)$ (normalized induction) where $C_c^{\infty}(\GL_1)$ is the regular representation of $\GL_1$ (and the induction is normalized). 

\end{itemize} 

\end{thm} 

Now we can  prove the following result which strengthens  Corollary \ref{C:theta1} and which is needed  later.  

\begin{prop}
\label{P:local_split} 
 Assume that we are in the $p$-adic case with $D$ split.
 Then $\Theta(1)$ is irreducible and isomorphic to $\Sigma$, 
the $N_D$-rank 2 representation of $G_D$ that appears as the  unique irreducible quotient of $I_D(1)$. 
\end{prop} 
\begin{proof} 
 Let $J(s)$ be the principal series for $\SL_2$ normalized so that the trivial representation is a quotient 
for $s=1$ and a submodule for $s=-1$. Likewise, let $J_D(s)$ denote the degenerate principal series associated to $Q_D$, normalized so that the trivial representation occurs at $J_D(\pm 5)$.
\vskip 5pt

 Let $\pi$ be an irreducible representation of $\SL_2$ and $\Theta(\pi)$ the corresponding big theta lift. 
Note that $\Theta(\pi)$ is always non-trivial, as a simple consequence of Lemma \ref{L:key}. Moreover, $\Theta(\pi)_{N_D,\psi_2}$ is isomorphic to $\pi^{\vee}$, so that it 
is infinite dimensional if and only if $\pi$ is. 
\vskip 5pt

Now if $\pi$ is a submodule of $J(-s)$, and $-s\neq 3,5$, then Theorem \ref{T:filtration} implies  by way of the Frobenius reciprocity  that $\Theta(\pi)$ is a quotient 
of $J_D(s)$. Now $J(-s)$  and $J_D(s)$ are irreducible for a generic $s$, in which case $\Theta(J(-s))\cong J_D(s)$.  It follows, from Lemma \ref{L:key},  that $J_D(s)_{N_D,\psi_2}$ is infinite dimensional for 
such $s$. Since the restriction of $J_D(s)$ to $N_D$ is independent of $s$, it follows that $J_D(s)_{N_D,\psi_2}$ is infinite dimensional for all $s$, in particular, for $s=1$. 
But $\Theta(1)$ is a quotient of $J_D(1)$ and $\Theta(1)_{N_D,\psi_2}$ is one-dimensional. Hence $\Theta(1)$ is isomorphic to the unique irreducible quotient of 
$J_D(1)$ which has $N_D$-rank 2. In particular, $\Theta(1)$ is irreducible and isomorphic to $\Sigma$, the unique quotient of $I_D(1)$.  
\end{proof} 

As a side remark, the representations $J_D(s)$ have $U_D$-rank 3. However, since $\Pi$ has $N_D$-rank 2, it follows that the two parabolic subgroups $P_D$ and $Q_D$ are 
not conjugated in $G_D$. But the two principal series $I_D(s)$ and $J_D(s)$ share all small rank subquotients: the trivial representation, the minimal representation and  the rank 2 representation $\Sigma$, as the above argument shows. 
\vskip 10pt

\section{Global lifting} 

Assume now that $F$ is a global field,  with its local completions denoted by $F_v$,  and  let $\mathbb A$ be the ring of ad\`eles over $F$. 
\vskip 5pt

\subsection{\bf Global theta lifting}
Let  $\Pi=\otimes \Pi_v$ be the restricted tensor product of minimal representations over all local places $v$ of $F$, where $\Pi_v\subset C^{\infty}(\Omega_v)$, 
as in Theorem \ref{T:model}. 
Every element in $\Pi$ is a finite linear combination of pure tensors $f= \otimes f_v$, where $f_v=f_v^{\circ}$ for almost all places $v$. 
There is a unique (up to a non-zero scalar) embedding $\theta: \Pi \rightarrow \mathcal A(G(F) \backslash G(\mathbb A))$ of $\Pi$ into the space of automorphic functions of uniform moderate growth. 
\vskip 5pt

We restrict $\theta(f)$ to the dual pair $D^1 \times G_D$ and for every $h\in \mathcal A(D^1(F)\backslash D^1(\mathbb A))$,  consider the function $\Theta(f,h)$ on $G_D$ defined by 
\[ 
\Theta(f,h)(g_D)= \int_{D^1(F)\backslash D^1(\mathbb A)} \theta(f)(g_D g) \cdot \bar h(g) ~dg.  
\] 
If this is to be of any use, we require the function $\theta(f)(g_D g) \cdot \bar h(g)$ to be of rapid decay on 
$D^1(F)\backslash D^1(\mathbb A)$  and  of moderate growth on  $G_D(F)\backslash G_D(\mathbb A)$. 
 This condition is clearly satisfied if $D^1$ is anisotropic or if  $h$ is a cusp form. It is also satisfied for a regularized theta lift, to be constructed in the next section. Namely,  for any finite place $v$,  we will construct an element $z$ in the Bernstein center of $\SL_2(F_v)$, such that for any $f\in \Pi$, the function 
$\theta(z\cdot f)(g_1 g)$  is of rapid decay on 
$D^1(F)\backslash D^1(\mathbb A)$  and  of moderate growth on  $G_D(F)\backslash G_D(\mathbb A)$.  
(See Proposition \ref{P:regularize}, and the discussion of this particular dual pair thereafter.) 
In particular, in all these cases,  the following integral is convergent: 
\[ 
\int_{N_D( \mathbb A) \backslash N_D(F)} \int_{D^1(F)\backslash D^1(\mathbb A)} |\theta(f)(n g) \cdot \bar h(g)| ~dgdn . 
\]

\vskip 10pt

\subsection{Fourier expansion} 
Let $\psi: \mathbb A/ F\rightarrow \mathbb C^{\times}$ be a non-trivial character. Then any $A\in J(F)$ defines a character 
$\psi_A$ of $N(F)\backslash N(\mathbb A)$ by $\psi_A(B)=\psi(\tr(A\circ B))$ for all $B\in N(\mathbb A) \cong J(\mathbb A)$.  
For every  $\varphi \in \mathcal A(G(F) \backslash G(\mathbb A))$, let 
\[  \varphi_A(g) = \int_{N(F)\backslash N(\mathbb A)} \varphi(ng) \cdot \overline{\psi_A(n)} ~dn \]
 be the Fourier coefficient  corresponding to $A$.   
We have a Fourier expansion 
\[ 
\theta(f)(g)= \theta(f)_0(g) + \sum_{A\in \Omega(F)} \theta(f)_A(g).
\]  
By uniqueness of local functionals, Theorem \ref{T:model},  for every $A\in \Omega(F)$ there exists a non-zero scalar $c_A$ such that 
\[ 
\theta(f)_A(g)= c_A \prod_v (g_v \cdot f_v)(A). 
\] 
This formula is particularly useful if $g_v\in M(F_v)$, for then $(g_v \cdot f_v)(A)= \chi_v(g_v) \cdot f_v(g_v^{-1} \cdot A)$ for some character $\chi_v$. 
\vskip 5pt

Let $\psi_2$ and $\psi_3$ be the rank 2 and 3 characters of $N_D(\mathbb A)$, as in the local case. Recall that $x\in \mathbb O$ is a pair $x=(y,z)$ of elements in 
$D$, and $N(x) = N(y) -\lambda N(z)$ for some $\lambda \in F^{\times}$. 
 Let $\varphi_{N_D, \psi_i}$ denote the global Fourier coefficient 
with respect to these two characters. Let $\omega_2(F)$ be the set of all rank one matrices 
\[ 
\left(\begin{array}{ccc} 
\pm 1 & x & 0 \\
-x & \pm 1 & 0 \\
0 &  0 & 0 
\end{array}\right) \in J(F) 
\]
such that  $x=(0,a)$ and $\lambda N(a) =\pm 1$ (for only one choice of sign, depending on $\psi_2$) i.e. the $2\times 2$ minor is 0.  
Then we have a global version of Lemma \ref{L:key}, 

\begin{lemma} \label{L:global} For every $f\in \Pi$, $\theta(f)_{N_D,\psi_3}=0$ and 
\[ 
\theta(f)_{N_D,\psi_2}(g)= \sum_{B \in \omega_2(F)} \theta(f)_B(g). 
\] 
\end{lemma}

  \vskip 10pt

\subsection{Non-vanishing of the theta lift}  We shall prove non-vanishing of the (regularized) theta lift by computing the Fourier coefficient 
\[ 
\Theta(f, h)_{N_D, \psi_2}(1)=
\int_{N_D( \mathbb A) \backslash N_D(F)} \int_{D^1(F)\backslash D^1(\mathbb A)} \theta(f)(n g) \cdot \bar h(g)\cdot \bar{\psi}_2(n)  ~dgdn .
\] 
Since this integral is absolutely convergent, we can reverse the order of integration. Then, using Lemma \ref{L:global}, we obtain 
\[ 
\Theta(f,h)_{N_D, \psi_2}(1)= \int_{D^1(F)\backslash D^1(\mathbb A))}\sum_{B\in \omega_2(F)} \theta(f)_{B}(g) \bar h(g) ~dg. 
\] 
Now fix \[ 
A=\left(\begin{array}{ccc} 
\pm 1 & x & 0 \\
-x & \pm 1 & 0 \\
0 &  0 & 0 
\end{array}\right) \in \omega_2(F) 
\]
where $x=(0,a)$, $a\in D$ such that $\lambda N(a)=\pm 1$. 

\begin{lemma}  \label{L:convergence} 
For every automorphic form $h$ and every $f\in \Pi$ we have 
\[ 
\int_{D^1(F)\backslash D^1(\mathbb A))}\sum_{B\in \omega_2(F)} \theta(f)_{B}(g) \bar h(g) ~dg = c_{A} \int_{D^1(\mathbb A)} f(g^{-1} A) \bar h(g) ~dg 
\] 
where the second integral is absolutely convergent. 
\end{lemma} 
\begin{proof} 
Since $\omega_2(F)$ is a principal homogeneous $D^1(F)$-space, the identity formally follows by unfolding the left hand side and using 
the formula for $ \theta(f)_{A}(g)$ as a product of local functionals, given above. Hence it remains to discuss the absolute convergence. 
\vskip 5pt

Let $f=\otimes_v f_v$. 
With $A$ as above, and $g\in \SL_2(F_v)$, $g^{-1} A$ is obtained from $A$ by replacing $x$ by $xg$.  As a quick explanation of the claimed absolute convergence, observe that 
$g\mapsto g^{-1}A$ is a closed embedding of $\SL_2(F_v)$ into $\Omega_v$, hence $g\mapsto f_v( g^{-1}A)$ is Schwartz function on $\SL_2(F_v)$.  
Since $h$ is of moderate growth the integral is clearly absolutely converging. 
\vskip 5pt

In more details,  assume firstly that $v$ is a finite place. Since $f_v \in \Pi_v$ is supported in a lattice in $J_v$, due to $N_v$-smoothness, 
it follows that  $g\mapsto f_v(g^{-1} A) $ is a compactly supported function on $\SL_2(F_v)$. Moreover, 
let $S$ be a finite set of places containing all archimedean places such that for $v\notin S$, all data is unramified: $D(F_v)$ is split,  $\lambda\in O_v^{\times}$,  $a\in \GL_2( O_v)$, 
$\psi_v$ has conductor $O_v$,  $f_v=f_v^{\circ}$, and $h$ is right $\SL_2(O_v)$-invariant. Here $O_v$ is the maximal order in $F_v$. 
It follows from Theorem \ref{T:WS} that $g\mapsto f^{\circ}_v(g^{-1} A) $ is the characteristic function of $\SL_2(O_v)$ for all $v\notin S$. Thus if we normalize the local measures 
so that $\vol( \SL_2(O_v))=1$ for all $v\notin S$, then 
\[ 
\int_{D^1(\mathbb A)} |f(g^{-1} A) \bar h(g)| ~dg = \int_{D^1(\mathbb A_S))} |f_S(g^{-1} A) \bar h(g)| ~dg 
\] 
where the subscript $S$ denotes the product of the local data over all places $v\in S$. 
\vskip 5pt

Consider now a real place. 
Let $g\in \SL_2(\mathbb R)$. 
 Assume that $g$ belongs to the double coset of the diagonal matrix 
 $\left(\smallmatrix t & 0 \\ 0 & 1/t \endsmallmatrix \right)$, $t>0$, in the Cartan decomposition of $ \SL_2(\mathbb R)$. Let $u=t+1/t$. 
  If we assume, for simplicity, that 
 $\lambda=1$, so $a$ in $x=(0,a)$ can be taken the identity matrix, then 
  $||xg||^2 = t^2 + 1/t^2$ (on the nose) and $||g^{-1}A||= t+1/t=u$. Presumably $N_{\infty}$-smoothnes of $f_{\infty}$ implies that every $f_{\infty}\in \Pi_{\infty}$ has a 
 rapid decrease on $\Omega_{\infty}$, and hence $g\mapsto f_v(g^{-1} A)$ has rapid decrease 
 on $\SL_2(\mathbb R)$, however, this statement is obvious for functions $f_{\infty}$ that we shall use, of compact support on $\Omega_{\infty}$ and the spherical 
 function, $f^{\circ}_{\infty}(g^{-1} A)= u^{-3/2} K_{3/2}(u)$, since the Bessel function $K_{3/2}(u)$ is rapidly decreasing as $u\rightarrow \infty$. 
 \end{proof} 
 
 We are now ready to prove the non-vanishing of the global theta lift. Assume firstly that $h$ is a cusp form. Then we have shown that 
\[ 
\Theta(f,h)_{N_D, \psi_2}(1)=  \int_{D^1(\mathbb A_S)} f_S(g^{-1} A) \bar h(g) ~dg 
\] 
for some large finite set of places.  Since for every $v\in S$ the local $f_v$ can be an arbitrary compactly supported smooth
function on $\Omega_v$ the integral will not vanish for some choice of data. Now consider the regularized theta integral 
$\Theta( z\cdot f, h)$, where $h$ is in an automorphic form, not necessarily cuspidal, and $z$ is an element of the Bernstein center of $\SL_2(F_v)$ (see the next section for the construction of $z$). The corresponding 
Fourier coefficient is 
\[ 
\Theta(z\cdot f,h)_{N_D, \psi_2}(1)=  \int_{D^1(\mathbb A)} (z\cdot f)(g^{-1} A) \bar h(g) ~dg.  
\] 
 Let $K_v$ be a sufficiently small open compact subgroup of  $\SL_2(F_v)$ such that $f_v$ is $K_v$-invariant.  Then 
 $z\cdot f_v=\alpha\cdot f_v$ where $\alpha$ is a  $K_v$ bi-invariant, compactly supported function on $\SL_2(F_v)$. 
  Let $\alpha^{\vee}(g)=\bar\alpha(g^{-1})$ and define $z^{\vee}\cdot h= \alpha^{\vee} \cdot h$. 
   Using the convergence guaranteed by Lemma \ref{L:convergence},
 \[
 \int_{D^1(\mathbb A)} (z\cdot f)(g^{-1} A) \bar h(g) ~dg = \int_{D^1(\mathbb A)} f(g^{-1} A) (\overline{z^{\vee}\cdot  h}) (g) ~dg, 
 \] 
  and this can again be arranged to be non-zero, provided $z^{\vee}\cdot  h\neq 0$. Hence we have proved the following: 
  
  \begin{theorem} \label{T:non_vanishing} 
  If $h$ is a non-zero cusp form on $D^1(\mathbb A)$, then $\Theta(f,h)\neq 0$ for some $f\in \Pi$. If $h$ is an automorphic form, not necessarily 
  cuspidal such that $z^{\vee}\cdot  h\neq 0$, then $\Theta(z\cdot f,h)\neq 0$ for some $f\in \Pi$. 
   \end{theorem}

\smallskip 
\noindent 
{\bf Remark:}  The main reason for introduction of the regularized theta lift is to be able to handle the lift of $h=1$ in the case when $D$ is split. In this case 
we can take all data to be simplest possible, i.e. $\lambda =1$,  the matrix $A$ with $a=(0,x)$ with $x$ identity matrix, etc. 
Then non-vanishing of the theta lift is achieved with the spherical vector $f^{\circ}_{\infty}$  at any real place.  Indeed, 
if $g\in \SL_2(\mathbb R)$  belongs to the double coset of the diagonal matrix 
 $\left(\smallmatrix t & 0 \\ 0 & 1/t \endsmallmatrix \right)$, $t>0$, in the Cartan decomposition of $ \SL_2(\mathbb R)$. 
 Then $||xg||^2 = t^2 + 1/t^2$ and $||g^{-1}A||= t+1/t$. Write $u=t + 1/t$ so that
 \[ 
 du= (t- \frac{1}{t}) \frac{dt}{t}.
 \] 
 Using the formula for the spherical vector given by Theorem \ref{T:DS} and the formula for the Haar measure on  $\SL_2(\mathbb R)$ with respect to the Cartan decomposition, we have 
 \[ 
\int_{ \SL_2(\mathbb R)} f_{\infty}^{\circ}(g^{-1} A) ~dg = \int_2^{\infty} u^{-3/2} K_{3/2}(u) ~du >0. 
 \] 
 It would be interesting to compute the value of this integral. 

\vskip 10pt

\section{Regularizing Theta} 

Following some ideas of Kudla and Rallis \cite{KR}, the first author  introduced in \cite{G} a regularized theta integral for a particular exceptional dual pair. 
We simplify the arguments so that regularization is now available to a 
wider class of examples. The notations used in this largely self-contained section will differ from those of the other sections of this paper.
We first recall some basic facts about the notions of uniform moderate growth and rapid decay.
\vskip 10pt

\subsection{Moderate growth and rapid decay} 
Let $k$ be a number field and let $\mathbb A$ denote the corresponding ring of adel\'es. Let $G$ be a reductive group over $k$. 
In order to keep notation simple, we shall assume that $G$ is split with a finite center. 
Fix a maximal split torus $T$ and a minimal parabolic subgroup $P$ containing $T$. Let $N$ be the unipotent radical of $P$. 
We have a root system $\Phi$, obtained by $T$ acting on 
the Lie algebra $\mathfrak g$ of $G$ and a set of simple roots in $\Phi$ corresponding to the choice of $P$.
\vskip 5pt

 If we fix a place $v$ of $k$, then $G_v$ will denote the group of 
$k_v$-points of $G$. Similarly, we shall use the subscript $v$ to denote various other subgroups of $G_v$. A smooth function $f$ on $G(\mathbb A)$ is of uniform moderate growth 
if there exists  an integer $m$ such that for every $X$ in the enveloping algebra of $\mathfrak g$ there exists a constant $c_X$ such that 
\[ 
|R_X f(g)| \leq c_X || g|| ^m 
\] 
where $R_X$ denotes the action of the enveloping algebra on smooth functions obtained by the differentiation from the right and $||g||$ is a height function on $G$ defined in
\cite[page 20]{MW}. 
Since there exists a constant $c$ such that $||gh|| \leq c ||g||\cdot   ||h||$ for all $g,h\in G(\mathbb A)$, it is easy to see that the constants $c_X$ 
for the right-translates $R_h f$ of $f$ are of moderate growth in $h$, more precisely, of growth $||h||^{m+d}$ where $d$ is the degree of $X$. 

\vskip 5pt 

Now assume that $v$ is a real or complex place of $k$. Let $P_v=M_v A_v N_v$ be the Langlands decomposition of $P_v$. For $\epsilon >0$, let 
 $A_{v,\epsilon} $ be a cone in $A_v$ consisting of $a\in A_v$ such that $\alpha(a) > \epsilon$ for all simple roots $\alpha$.  
 Let $A$ be the product of the $A_v$'s  and  let $A_{\epsilon}$ be the product of the $A_{v,\epsilon}$'s  over all 
 real and complex places $v$. Let $\omega_N$ be a compact set in $N(\mathbb A)$ containing the identity element. Let $K$ be a product of maximal compact subgroups 
 $K_v$ of $G_v$ where we have taken $K_v$ to be hyperspecial for all $p$-adic places. Then 
 \[ 
 S=\omega_N  A_{\epsilon} K 
 \] 
 is a Siegel domain in $G(\mathbb A)$. If $\omega_N$ is sufficiently large, and $\epsilon$ is sufficiently small, then $G(\mathbb A)= G(k) S$. 

\vskip 5pt

Let $\Pi$ be an automorphic representation of $G$. Then any smooth $f\in \Pi$ is of uniform moderate growth. 
In terms of the Siegel domain $S$,  this means the following.  Let $\rho_P : A\rightarrow \mathbb R^+$ be the modular character.
There exists an integer $m$ such that for every $X$ in the enveloping algebra of $\mathfrak g$, there exists a constant $c_X$ such that 
\[ 
|R_X f(nak)| \leq c_X \cdot  \rho_P(a)^m 
\] 
on $S$, where the constants $m$ and $c_X$ are not necessarily the same, but related to those above. 

\vskip 5pt

Now let $Q\supseteq P$ be a maximal parabolic with a unipotent radical $U\subseteq N$, corresponding to a simple root $\alpha$. We have a standard 
 Levi factor $L$ of $Q$ defined as the centralizer of a fundamental co-character $\chi : \mathbb G_m \rightarrow T$ (or a power of it). In any case, any element in 
 $A_v$ is uniquely written as a product $\prod_{\chi} \chi(t_{\chi})$, over all fundamental co-characters $\chi$, where $t_{\chi}\in \mathbb R^+$.  
 The element  $\prod_{\chi} \chi(t_{\chi})$ is contained in the cone $A_{v,\epsilon}$ if $t_{\chi}> \epsilon$ for all $\chi$. 
 Let $f_U$ be the constant term of $f$ along $U$. Then, if $f$ has a uniform moderate growth, by   \cite[page 30, Lemma]{MW} 
 for every positive  integer $i$,  there is a constant $c_i$ such that 
\[ 
| (f-f_U)(nak)| \leq  c_i \cdot \rho_P(a)^m \cdot  \alpha^{-i}(a) 
\] 
on $S$. 
In particular, if $f_U=0$,   then $f$ is rapidly decreasing in the variable $t_{\chi}$.
 If $f_U=0$ for all maximal parabolic subgroups, then $f$ is rapidly decreasing on $S$, and that's how the rapid decrease of cusp forms is established. The proof of 
 \cite[page 30, Lemma]{MW} 
 involves integration by parts, so it is easy to see that the constants $c_i$ for the right-translates $R_h f$ of $f$ are of moderate growth in $h$, more precisely, 
 of the growth $||h||^{m'}$ where $m'$ depends on $i$:  a larger $i$ will demand a larger $m'$. 
 
 \smallskip 
 We highlight another important issue here.  Assume that $f$ belongs to an automorphic representation $\pi$. Then a Frech\' et space topology on $\pi$ is given by the family of 
 semi-norms 
 \[ 
 ||f||_X = \sup_{nak\in S} |R_X f(nak)|  \cdot \rho_P(a)^{-m}
 \] 
 where $m$ depends on $\pi$ and works for all $X$ in the enveloping algebra. Then \cite[page 30, Lemma]{MW} says that convergence in these seminorms implies 
 convergence in the seminorm  
 \[ 
 \sup_{nak\in S} |(f-f_U)(nak)| \cdot  \rho_P(a)^{-m} \cdot \alpha^{i}(a).
 \] 
  This observation will later imply that the regularized theta integral gives a continuous pairing. 
\vskip 10pt

\subsection{Restricting to a subgroup}

Let $G_1 \times G_2 \subseteq G$ a dual pair in $G$.  Let $T_1$ be a maximal split torus in $G_1$ and fix a minimal parabolic subgroup $P_1$ containing $T_1$. 
Without loss of generality,  we can assume that $T_1\subseteq T$ and $P_1\subseteq P$. Let $Q_1\supseteq P_1$ be a maximal parabolic subgroup of $G_1$. Let 
$\chi_1: G_m \rightarrow T_1$ be the corresponding fundamental cocharacter (or a multiple of which) so that the centralizer of $\chi_1$ in $G_1$ is a Levi factor $L_1$ of $Q_1$.  
Assume that: 
\vskip 5pt

\noindent \underline{\bf  Hypothesis}:
{\it For every fundamental cocharacter $\chi_1$ of $G_1$, there is a fundamental cocharacter $\chi$ of $G$  such that $\chi_1$ is a multiple of $\chi$.}
 
\vskip 10pt

This hypothesis holds in the following examples:
\vskip 5pt

\begin{itemize}
\item  the dual pair $G_1 \times G_2 = D^1 \times G_D = \SL_2 \times G_D$ studied in this paper;  here $G_1=\SL_2$ corresponds to the highest root and  the highest weight is also a fundamental weight  for $\mathbf E_7$ (the ambient group $G$). 
\vskip 5pt

\item the split exceptional dual pairs in $G$ of type $\mathbf E_n$ where one member of the dual pair is the type $\mathbf G_2$, see \cite{LS}. In particular, this includes the case $\PGL_3 \times G_2$ treated in \cite{G}. 
\end{itemize}

\vskip 5pt

The above hypothesis have the following consequences:
\vskip 5pt

\begin{itemize}
\item It implies that the cone $A_{1,\epsilon}$ sits as a subcone of $A_{\epsilon}$; in fact, it is a direct factor in the above cases. In particular, we have an inclusion of Siegel domains $S_1 \subset S$.

 \vskip 5pt
 
 \item Given a fundamental cocharacter $\chi_1$ of $G_1$, the associated fundamental cocharacter $\chi$ of $G$ given by the hypothesis corresponds to a simple root and so determines a maximal parabolic subgroup  $Q_{\chi_1}  = L_{\chi_1}U_{\chi_1}$ of $G$.  In the following, we will sometimes write $U = U_{\chi_1}$ to simplify notation. 
 \end{itemize}
\vskip 5pt
 
Now let $v$ be a $p$-adic place and $z$ an element of the Bernstein's center of $G_1(k_v)$. Then $z\cdot \Pi$ is naturally a $G_1(\mathbb A)\times G_2(\mathbb A)$-submodule of 
$\Pi$.  For a fixed cocharacter $\chi_1$ of $G_1$, with associated maximal parabolic $Q = LU$, assume that  
\[  z \cdot \Pi_v  \subset {\rm Ker}\left( \Pi_v \longrightarrow (\Pi_v)_{U(k_v)} \right). \]
We claim that this implies  that $(z\cdot f)_U=0$ on  $G_1(\mathbb A)\times G_2(\mathbb A)$. 
 Indeed, if $g\in G_1(\mathbb A)\times G_2(\mathbb A)$, then 
\[ 
(z\cdot f)_U(g) = (R_g(z\cdot f))_U (1) = (z\cdot R_g(f))_U(1)=0 
\] 
where $R_g$ denotes the right translation by $g$. Here, the second equality holds since $z$ and $R_g$ commute, and the third equality holds since the projection of $z\cdot \Pi$ on $\Pi_U$ vanishes. 
Write $g=g_1\times g_2 \in G_1(\mathbb A) \times G_2(\mathbb A)$ and assume that $g_1\in S_1$. Using the hypothesis that $S_1 \subseteq S$ and the estimates for 
$|R_{g_2}(z\cdot f)- (R_{g_2}(z\cdot f))_U|$ on $S$ from the last subsection, it follows that  
\[ 
(z\cdot f)(g_1\times g_2)= R_{g_2} (z\cdot f) (g_1) 
\] 
is of moderate growth in both variables and  in the variable $g_1\in S_1$,  it is rapidly decreasing in the direction of the fundamental co-character $\chi_1$. More precisely, 
we summarise the discussion in this subsection in the following proposition.
\vskip 5pt

\begin{prop} \label{P:regularize} 
Assume that:
\vskip 5pt

\begin{itemize}
\item[(i)]  For every fundamental cocharacter $\chi_1$ of $G_1$, there is a fundamental cocharacter $\chi$ of $G$  such that $\chi_1$ is a multiple of $\chi$, which in turn determines a maximal parabolic subgroup $Q_{\chi_1} = L_{\chi_1} U_{\chi_1}$;
\vskip 5pt

\item[(ii)]  One can find an element $z$ in the Bernstein center of $G_1(k_v)$ such that for every fundamental cocharacter $\chi_1$ of $G_1$,   the natural projection $\Pi_v$ on $(\Pi_v)_{U_{\chi_1}(k_v)}$ vanishes on $z\cdot \Pi_v$ for every fundamental co-character $\chi_1$ of $G_1$. 
 \end{itemize}
\vskip 5pt

\noindent Then for every integer $n$, there exists an integer $m$ and a constant $c$ such that 
\[ 
|(z\cdot f)(g_1\times g_2)|\leq c || g_1 ||^{-n} || g_2||^m 
\] 
for all $g_1 \in S_1$ and $g_2 \in G_2(\mathbb A)$.
\end{prop}
\vskip 5pt

 In the context of the above proposition, a small trade-off here is that increasing $n$ can be obtained only by increasing $m$ at the same time. But this is still good enough to define 
regularized theta lift which produces functions of moderate growth as output. 
To exploit the proposition, 
it remains then to construct an appropriate $z$. We also need to assure that $z\cdot\Pi_v\neq 0$ and this may not be possible always, as it will be discussed in the next subsection. 
\vskip 10pt

\subsection{Bernstein's center}

We work here locally over a $p$-adic field. Thus all our groups are local and we drop the subscript $v$. For simplicity, we shall discuss only the Bernstein center for the Bernstein component  containing the trivial representation of $G_1$. 

\vskip 5pt

To that end, let  $\hat{T}_1$ be the complex torus dual to $T_1$, and let $W(G_1)$ be the Weyl group of $G_1$. 
The Bernstein's center $Z(G_1)$ of the said component is isomorphic to 
the algebra of $W(G_1)$-invariant regular functions on $\hat{T}_1$. Similarly, the Bernstein's center  $Z(L_1)$ of the Levi factor $L_1$ is isomorphic to 
the algebra of $W(L_1)$-invariant regular functions on $\hat{T}_1$. In particular, we have a natural map $j :  Z(G_1) \rightarrow Z(L_1)$. Let $\pi$ be a smooth representation 
of $G_1$, and  let $p: \pi \rightarrow \pi_{U_1}$ be the natural projection onto the normalised Jacquet module $\pi_{U_1}$. Then, for every $z\in Z(G_1)$ and $v\in \pi$, we have 
\[ 
p(z \cdot v)= j(z) \cdot (p(v)).  
\] 
\vskip 5pt

Now let $\Pi$ be a smooth representation of $G$. Recall that we want to find a non-zero $z\in Z(G_1)$ such that $z\cdot \Pi$ is in the kernel of the projection of $\Pi$ onto $\Pi_{U}$. 
Since $\Pi_{U}$ is a quotient of $\Pi_{U_1}$, and $j(z)$ is acting on $\Pi_U$,  
 we need to find $z$ such that $j(z)=0$ on $\Pi_U$.  This is always possible if $\Pi$ is a finite length $G$-module, in which case
 $\Pi_U$ is a finite length $L$-module. In particular, the center of $L$ acts finitely on $\Pi_U$. Hence, the center of $L_1$ (= the center of $L$) acts finitely on $\Pi_U$ and
the  $Z(L_1)$-spectrum of $\Pi_U$ is contained in a proper subvariety of $\hat{T}_1$. In particular, any non-zero $W(G_1)$-invariant function $z$ vanishing on the subvariety will have the desired property that $j(z)$ vanishes on $\Pi_U$.  Hence,  a non-trivial $z$ with the desired property always exists. 
\vskip 5pt

A potential trouble is that such a $z$ may kill the whole $\Pi$.  
However, if $G_1$ is a smaller group (still split) and 
$\Pi$ the minimal representation, then every Whittaker generic representation of $G_1$ appears as a quotient of $\Pi$ (at least int he family of examples we have in mind). Hence, if $z$ kills $\Pi$, then $z$ kills all 
generic representations of $G_1(k_v)$ and hence $z$ must be equal to 0.   Therefore the desired regularization can be carried out in this case. 

\vskip 5pt 
Let's look at our dual pair $G_1\times G_2= \SL_2  \times G_D$ in $G$, and $\Pi$ is the minimal representation. 
 The Bernstein's center is 
\[ 
Z(G_1)=\mathbb C[x^{\pm1}]^{S_2}
\] 
where $S_2$ acts by permuting $x$ and $x^{-1}$. Let 
\[ 
z= (x -q^3)(x^{-1} -q^3) (x -q^5)(x^{-1} -q^5)
\] 
where $q$ is the order of the residual field. This element satisfies our requirement, since $j(z)$ vanishes on $\Pi_U$ by Theorem \ref{T:filtration}. 

\subsection{Global $\Theta(1)$} 

Let $z$ be the element in the Bernstein center of $G_1 = \SL_2$, as in the previous subsection. We define $\Theta(1)$ as the space of automorphic functions, $g_D\in G_D(\A)$, 
\[ 
\Theta(f)(g_D)=\int_{D^1(F)\backslash D^1(\mathbb A)} \theta(z\cdot f)(g_D g)  ~dg.
\]
where we assume that $f_{\infty}$ is $K_{\infty}$-finite. (We assume this finiteness since in the next section we will determine the local lift at real places in the 
language of $(\mathfrak g,K)$-modules.) We want to show that $\Theta(1) \neq 0$, using Theorem  \ref{T:non_vanishing}. The input in the theta kernel is $h=1$, 
so the first thing is to show that $z^{\vee}\cdot 1 \neq 0$. In the case at hand, $z^{\vee}$ is obtained from $z$ by replacing $x$ by $x^{-1}$ in the above expression of 
$z$. In particular, $z=z^{\vee}$.  Moreover, $z$ acts on the trivial representation by the scalar obtained by substituting $x=q$, and this is non-zero. It remains to argue that 
we can arrange $f_{\infty}$ to be $K_{\infty}$-finite. This follows by  the continuity of the regularized theta integral, which ensures that  the non-vanishing for smooth $f$ 
implies the non-vanishing for $K_{\infty}$-finite vectors. Alternatively, by the remark following Theorem \ref{T:non_vanishing}, non-vanishing can be achieved 
with $f_{\infty}=f^{\circ}_{\infty}$ the spherical vector.

 \section{Correspondence for real groups} 
 
 In this section, we work over the field $\R$ of real numbers.  The goal of this section is to determine $\Theta(1)$ explicitly. For this,
 we need to consider  various cases separately.   
 Indeed, recall that $G$ is arising from an Albert algebra via the Koecher-Tits construction. There are two real forms of 
 octonion algebra, the classical Graves algebra and its split form, and these two algebras can be used to define two Albert algebras of $3\times 3$-hermitian symmetric 
 matrices with coefficients in the octonion algebra.  The group $G$ is split or of the relative rank $3$  depending on whether the octonion algebra is split or not. 
 Moreover,  it will be  convenient to work with the simply connected cover of $G$ and the $(\mathfrak g, K)$-module corresponding to the minimal representation. 

 \vskip 5pt

\subsection{\bf Non-split $\mathbb O$.}
  Assume first that $\mathbb O$ and hence $G$ is not split. Then the minimal representation, when restricted to the simply connected cover 
 (simply connected in the sense of algebraic groups) breaks up as $\Pi=\Pi_{1,0} \oplus \Pi_{0,1}$, a sum of a holomorphic and an anti-holomorphic irreducible representation. 
 This sum is the socle of the degenerate   principal series $I(-5)$.  
 
 \vskip 5pt
 
 In this case, $D$ is necessarily nonsplit.
 The socle of $I_D(-1)$, considered a representation of the simply connected cover of $G_D$ is a direct sum of three 
 representations $\Sigma_{2,0} \oplus \Sigma_{1,1} \oplus \Sigma_{0,2}$, a holomorphic, a spherical and an anti-holomorphic representation, respectively. 
 One has:
 \vskip 5pt
 
 \begin{thm} \label{T:real1}
  If $\mathbb O$ is non-split (so $G$ is not split), we have 
 \[ 
 \Pi_{1,0}^{D^1}\cong  \Sigma_{2,0} \quad \text{ and } \quad  \Pi_{0,1}^{D^1}\cong \Sigma_{0,2}. 
 \] 
 In particular, $\Theta(1) = \Sigma_{2,0} \oplus \Sigma_{0,2}$.
 \end{thm}
 In view of the map $\Theta(1) \longrightarrow I_D(-1)$, the above identities are established by checking that the $K_D$-types coincide, and this is an easy check that we shall omit. 
 
 \smallskip 
 
 \subsection{\bf Split $\mathbb O$ but nonsplit $D$}
 
 We move on the case when $\mathbb O$ and hence $G$ is split. Let $K$ be a maximal compact subgroup of $G$, and 
 $\mathfrak g = \mathfrak k \oplus \mathfrak p$ the corresponding  Cartan decomposition of the complexification of the Lie algebra of $G$. Then 
  $\mathfrak k$ is isomorphic to $\mathfrak{sl}_8$. Fixing this isomorphism, we see that as a $K\cong \SU_8/\mu_2$-module, $\mathfrak p$  is isomorphic to $V_{\omega_4}$, where 
  $\omega_4$ is the 4-th fundamental weight.  The minimal representation $\Pi$ is a direct sum of $K$-types $V_{n\omega_4}$, where $n=0,1,2, \ldots $. 
 \vskip 5pt

 We have two cases depending on $D$. 
 Assume in this subsection that $D$ is a division algebra. In this case $D^1\cong \SU_2$ is compact, and embeds into $\SU_8$ as a $2\times 2$ block. The centralizer of $\SU_2$ in 
$K=\SU_8/\mu_2$ is  $K_D\cong \U_6$. The minimal representation $\Pi$ decomposes discretely when restricted to this dual pair. A simple application of the 
Gelfand-Zetlin rule shows that the $K_D$-types of $\Theta(1)$ are multiplicity free and the highest weights of the $K_D$-types which occur are 
\[ 
(x,x,0,0,y,y) 
\] 
where $x\geq 0 \geq y$ are any two integers. Here we are using the standard description of highest weights for $\U_6$ by 6-tuples of non-increasing integers. 
But these are precisely the $K_D$-types of the spherical submodule of $I_1(-1)$,  i.e. the constituent $\Sigma_{1,1}$  in \cite[Theorem C]{Sa93}. This proves 
\vskip 5pt

\begin{thm}  \label{T:real2}  
When $\mathbb O$ is split but $D$ is non-split, one has
 \[ 
\Theta(1)=  \Pi^{D^1}\cong  \Sigma_{1,1}. 
 \] 
 \end{thm}
\vskip 10pt

 \subsection{Split $\mathbb O$ and split $D$.} 
 This is the most involved case. 
 Let $(e,h,f)$ be an $\mathfrak{sl}_2$-triple spanning the complexified Lie algebra of $D^1=\SL_2$. After conjugating by $G$, 
  if necessary, we can assume that the triple is  
  stable under the Cartan involution. Then $e\in \mathfrak p$ is a highest weight vector for the action of $K$,  and $h\in \mathfrak k$. 
  Let $\Theta(1)$ be the maximal quotient of the $(\mathfrak g,K)$-module of the minimal representation such that the $\mathfrak{sl}_2$ triple 
  acts trivially. 

\begin{thm} \label{T:Siegel-Weil} 
$\Theta(1)$ is irreducible and isomorphic to the unique submodule $\Sigma$ of $I_D(-1)$, which is a spherical representation. 
\end{thm} 

  The proof of this result will take the rest of this section. 
 After conjugating by $K$, if necessary, we can assume that  
 \[ 
  h = \frac12
  \left( \begin{array}{cccccccc} 
  1 & & & & & & &  \\ 
  & 1 & & & & & &   \\ 
&& 1 & & & & &   \\ 
&&& 1 & & & &   \\ 
&&&& -1 & & &   \\ 
&&&&& -1 & &   \\ 
&&&&&& -1 &  \\ 
&&&&&&& -1  \\ 
\end{array}\right) \in \mathfrak{sl}_8. 
\] 
Let $G_1$ be the centralizer of the $\mathfrak{sl}_2$-triple in $G$. It is a group isomorphic to $\Spin(6,6)$.
 Let $\mathfrak g_1 = \mathfrak k_1 \oplus \mathfrak p_1$ be the corresponding  Cartan decomposition. Then $\mathfrak k_1 \cong \mathfrak{sl}_4\oplus \mathfrak{sl}_4$ 
 sitting block diagonally in $\mathfrak{sl}_8$. The centralizer of $h$ in $\SU_8/\mu_2$ is  
 \[ 
 K_1=\SU_4\times\SU_4/\Delta\mu_2
 \] 
  and this confirms that $G_1$ is simply connected (as 
 a group in a given (non-hermitian) isogeny class is determined by its maximal compact subgroup). 
 
 \vskip 5pt 
 
 Let $\Pi$ be the $(\mathfrak g,K)$-module corresponding to the minimal representation of $G$. Then, as a $K$-module,  
 \[ 
 \Pi = \oplus_{n\geq 0} V_{n\omega_4}. 
 \] 
 We shall also need the following facts about the action of $e$ on $\Pi$.  From the formula for the tensor product $V_{\omega_4}\otimes V_{\n\omega_4}$ it follows that 
 \[ 
 e \cdot V_{n\omega_4}  \subseteq V_{(n-1) \omega_4} \oplus V_{(n+1) \omega_4}. 
 \] 
 Since $\Pi$ is not a highest weight module, by    \cite[Lemma 3.4]{V}, $e$ is injective on $\Pi$. The same results hold for $f$. 
 
 \vskip 5pt 
 
 Let $\pi$ be an irreducible $\mathfrak{sl}_2$-module such that $h$ acts semi-simply and integrally. Let $\Theta(\pi)$ be the 
 big theta lift of $\pi$; it is a $(\mathfrak{g}_1, K_1)$-module.   We shall now partially determine the structure of $K_1$-types of $\Theta(\pi)$. In order to state the result, we need some additional notation. 
 A highest weight $\mu$ for $\SU_4$ is represented by a quadruple $(x,y,z,u)$ of integers, such that $x\geq y \geq z \geq u$, and it is determined by the triple 
 \[  \text{$\alpha=x-y$, $\beta=y-z$, $\gamma=z-u$} \]
  of non-negative integers. 
 \vskip 5pt
 
 \begin{prop} \label{P:K-types} 
  Let $V\otimes U$ be a $K_1\cong \SU_4 \times \SU_4/\Delta \mu_2$-type of $\Theta(\pi)$. 
  Then $U \cong V^*$, the dual representation of $V$, and the multiplicity of $V \otimes V^*$ in 
 $\Theta(\pi)$ is at most one. If $\pi=1$, the trivial representation, and $\mu$ is the highest weight of $V$, then $\alpha=\gamma$. 
 \end{prop} 
 \vskip 5pt
 
\begin{proof} 
We need the following lemma which can be easily deduced from the Gelfand-Zetlin branching rule. 

\begin{lemma} The restriction of $V_{n\omega_4}$ to $\mathfrak{sl}_3 \oplus \mathfrak{sl}_3 \oplus \mathbb C h$ is multiplicity free and given by 
\[ 
V_{n\omega_4}= \oplus_{n\geq x\geq y \geq z \geq u\geq 0} V_{\mu} \otimes V_{\mu}^{\ast} \otimes \mathbb C(m) 
\] 
where $\mu$ is represented by the  quadruple $(x,y,z,u)$ and $h$ acts on $\mathbb C(m)$ by the integer $m=x+y+z+u-2n$. 
\end{lemma} 
\vskip 5pt

It follows from the lemma that the only $K_1$-types appearing in the restriction of $\Pi$ are isomorphic to $V\otimes V^*$, as claimed. In order to prove multiplicity one 
in $\Theta(\pi)$, we proceed as follows.
\vskip 5pt

 Let $m$ be an integer appearing as an $h$-type in $\pi$. Let $\Omega$ be the Casimir element 
for $\mathfrak{sl}_2$ and let $\chi : \mathbb C[\Omega] \rightarrow \mathbb C$ be the central character of $\pi$. Let $\Pi(\mu, m)$ be the maximal subspace of $\Pi$ such that 
$h$ acts as the integer $m$ and $\mathfrak{sl}_3 \oplus \mathfrak{sl}_3$ as a multiple of $V_{\mu} \otimes V_{\mu}^*$.  Note that $\Pi(\mu, m)$  is naturally a 
$\mathbb C[\Omega] $-module, and it suffices to show that the maximal quotient of $\Pi(\mu, m)$ such that $\mathbb C[\Omega] $ acts on it by $\chi$ is isomorphic to 
$V_{\mu} \otimes V_{\mu}^{\ast}$ as an $\mathfrak{sl}_3 \oplus \mathfrak{sl}_3$-module. We have a canonical isomorphism 
\[ 
\Pi(\mu, m) \cong (V_{\mu} \otimes V_{\mu}^{\ast} ) \otimes \Hom_{\mathfrak{sl}_3 \oplus \mathfrak{sl}_3} (V_{\mu} \otimes V_{\mu}^{\ast}, \Pi(m)) 
\] 
and $\mathbb C[\Omega] $ acts on  
\[ 
\Hom_{\mathfrak{sl}_3 \oplus \mathfrak{sl}_3} (V_{\mu} \otimes V_{\mu}^{\ast}, \Pi(m)) = 
\oplus_{n\geq 0} \Hom_{\mathfrak{sl}_3 \oplus \mathfrak{sl}_3} (V_{\mu} \otimes V_{\mu}^{\ast}, V_{n\omega_4}( m)) 
\] 
Now notice that, given $\mu$ and $m$, $ \Hom_{\mathfrak{sl}_3 \oplus \mathfrak{sl}_3} (V_{\mu} \otimes V_{\mu}^{\ast}, V_{n\omega_4}( m)) \neq 0$ for only one parity of 
$n$. Furthermore, if this space is non-zero for some $n$, then it is non-zero for $n+2$, as $\mu$ is also represented by $(x+1,y+1,z+1,u+1)$ and 
\[ 
m=x+y+z+u-2n=x+1+y+1+z+1+u+1-2(n+2). 
\] 
Let $n_0$ be the first integer such that $\Hom_{\mathfrak{sl}_3 \oplus \mathfrak{sl}_3} (V_{\mu} \otimes V_{\mu}^{\ast}, V_{n_0\omega_4}( m)) \neq 0$ and let 
$T_0$ be a generator of this one-dimensional space. We then have a natural map 
\[ 
A: \mathbb C[\Omega] \cdot T_0 \rightarrow  \Hom_{\mathfrak{sl}_3 \oplus \mathfrak{sl}_3} (V_{\mu} \otimes V_{\mu}^{\ast}, \Pi(m)). 
\] 
\vskip 5pt 
\begin{lemma}
The  map $A$ is an isomorphism. 
\end{lemma}
\vskip 5pt

\begin{proof}
Let $i$ be a non-negative integer. 
Let $\mathbb C[\Omega]_i$ be the space of polynomials of degree less then equal to $i$, and let 
\[ 
\Hom_{\mathfrak{sl}_3 \oplus \mathfrak{sl}_3} (V_{\mu} \otimes V_{\mu}^{\ast}, \Pi(m))_i= 
\oplus_{j=0}^i \Hom_{\mathfrak{sl}_3 \oplus \mathfrak{sl}_3} (V_{\mu} \otimes V_{\mu}^{\ast}, V_{(n_0 +2j)\omega_4}(m)).
\] 
These two spaces have dimension $i+1$ and define filtrations of $\mathbb C[\Omega]$   and 
$\Hom_{\mathfrak{sl}_3 \oplus \mathfrak{sl}_3} (V_{\mu} \otimes V_{\mu}^{\ast}, \Pi(m))$ as $i$ increases. 
Since $\Omega$ has degree 2, as an element of the enveloping algebra of $\mathfrak g$, the map $A$ preserves the two filtrations. Thus, in oder to prove the claim, 
it suffices to show that $A$ is injective.
\vskip 5pt

 But if it is not, then there would be a polynomial $p(\Omega)$ acting trivially on 
$V_{\mu}\otimes V_{\mu}^{\ast} \subseteq V_{n_0\omega_4}(m)$. Under the action of $\mathfrak{sl}_2$, this subspace would generate a finite length representation of 
$\mathfrak{sl}_2$ and hence the whole $\Pi$ would be discretely decomposable under the action of $\mathfrak{sl}_2$.  This gives a contradiction  and the Lemma is proved. 
\end{proof}
\vskip 5pt 

The Lemma implies that 
\[ 
\Pi(\mu,m) \cong (V_{\mu} \otimes V_{\mu}^{\ast} ) \otimes \mathbb C[\Omega] 
\] 
as $\mathbb C[\Omega]$-modules. Hence, if we fix a character $\chi$ of $\mathbb C[\Omega]$, the maximal quotient of $\Pi(\mu,m)$ such that $\mathbb C[\Omega]$
acts by $\chi$ is isomorphic to $V_{\mu} \otimes V_{\mu}^{\ast}$. This proves that $\Theta(\pi)$ has multiplicity free $K_1$-types. 

\vskip 5pt

We proceed to narrow down the $K_1$-types appearing in $\Theta(1)$.  For every $\mu$, the action of $e$ on $\Pi$ gives an injective map 
\[ 
e : \Pi(\mu, -2) \rightarrow \Pi(\mu,0). 
\] 
\begin{lemma} If $e: \Pi(\mu, -2) \rightarrow \Pi(\mu,0)$ is bijective, then $V_{\mu} \otimes V_{\mu}^{\ast}$ is not a $K_1$-type of $\Theta(1)$. 
\end{lemma} 
\begin{proof} 
The image of $e$ is necessarily contained in the kernel of the natural surjective map  $\Pi(\mu,0) \rightarrow \Theta(1)(\mu)$. Hence the lemma follows. 
\end{proof} 
\vskip 5pt

Consider the filtration $\Pi(\mu, m)_i= \oplus _{n\leq i} V_{n\omega_4}(\mu, m)$ of $\Pi(\mu,m)$. Then we have an injective map 
\[ 
e : \Pi(\mu, -2)_i \rightarrow \Pi(\mu,0)_{i+1} 
\] 
for all $i$. Hence, if the dimensions of the two spaces are equal  for all $i$, then $e$ is bijective. This will happen precisely when $V_{\mu}\otimes V_{\mu}^{\ast}$ 
occurs in $V_{n\omega_4}(-2)$ but not in $V_{(n-1)\omega_4}(0)$, for some $n$.  The occurrence in $V_{n\omega_4}(-2)$ implies that there exists a unique 
quadruple $(x,y,z,u)$ representing $\mu$ such that 
\[ 
n \geq x \geq y \geq z \geq u \geq 0 \text{ and } x+y+z+u -2n=-2 
\] 
Then $V_{\mu}\otimes V_{\mu}^{\ast}$ occurs in $V_{(n-1)\omega_4}(0)$ if and only if 
\[ 
n-1 \geq x \geq y \geq z \geq u \geq 0 \
\] 
i.e. $n>x$. Thus, if $n=x$, then $V_{\mu} \otimes V_{\mu}^{\ast}$ does not appear in $\Theta(1)$.
\vskip 5pt

 Let's see what this means in terms of $\alpha$, $\beta$ and $\gamma$. 
We have to find $n$ such that $\mu$ is represented by 
\[ 
(x,y,z,u)= (n, n-\alpha, n-\alpha-\beta, n-\alpha-\beta-\gamma). 
\] 
Since the last entry must be non-negative, we have $n\geq \alpha+ \beta+\gamma$. On the other hand, $h$ has to act as $-2$, hence 
$x+y+z+u -2n=-2$ and this is equivalent to $2n=3\alpha+2\beta+ \gamma -2$. Combining with the previous inequality, we obtain $\alpha\geq \gamma +2$. 
Hence the types with $\mu$ such that $\alpha\geq \gamma+2$ do not appear in $\Theta(1)$. Replacing the role of $e$ with $f$, a similar argument shows that 
the types such that $\gamma\geq \alpha +2$ do not appear either. Hence $|\alpha-\gamma|\leq 1$  for all types that appear in $\Theta(1)$. Since 
$\alpha \equiv \gamma \pmod{2}$  for any type in $\Pi(0)$, $\alpha=\gamma$ for all types that appear in $\Theta(1)$. This completes the proof of Proposition \ref{P:K-types}.
\end{proof} 

The types of $\Theta(1)$, as described in Proposition \ref{P:K-types}, are the same as the types of the spherical rank-2 submodule in $I_D(-1)$ by  \cite[Theorem 4B]{Sa95}. 
This proves Theorem \ref{T:Siegel-Weil}. 
\vskip 5pt

\section{Siegel-Weil formula and Consequences} 

We are now ready to prove the main results of this paper (Theorem \ref{T:main} and Theorem \ref{T:sw} in the introduction).
Assume that $F$ is a totally real global field and $D$ a quaternion algebra over $F$.  
\vskip 5pt

\subsection{The representation $\Theta(1)$}
We have shown that the global (regularized) theta lift $\Theta(1)$ is a non-zero automorphic 
representation of $G_D(\mathbb A)$. We have also studied the abstract local theta lift of the trivial representation of $D^1$ to $G_D$.  
The following summarizes what we have shown:
\vskip 5pt

\begin{prop}   \label{P:Theta1}
\noindent (i) The automorphic representation $\Theta(1)$ is irreducible and occurs with multiplicity one in the space of automorphic forms of $G_D$;
\vskip 5pt

\noindent (ii) For every $p$-adic place $v$ of $F$, the local component $\Theta(1)_v$ is isomorphic to the 
unique irreducible quotient of the local degenerate principal series $I_D(1)$. 
 \vskip 5pt
 
 \noindent (iii) For every real place $v$ of $F$, the local component $\Theta(1)_v$ is an irreducible quotient of $I_D(1)$ as described in Theorems \ref{T:real1}, \ref{T:real2} and \ref{T:Siegel-Weil}.
 \end{prop} 
\vskip 5pt

\begin{proof}  Indeed, we have shown that the abstract local theta lift $\Theta(1_v)$ is irreducible. Hence the global $\Theta(1)$ is an irreducible automorphic representation. 
The fact  that $\Theta(1)$ has multiplicity one in the space of automorphic forms follows by  \cite[Theorem 1.1]{KS15}. 
Note that the required conditions, as spelled out in the introduction of \cite{KS15}, 
are satisfied by the recent work of M\"ollers and Schwarz \cite{MS17}. 
\end{proof}

\vskip 10pt

\subsection{A Siegel-Weil formula}
For a flat section $\Phi \in I_D(s)$, let $E_D(s, \Phi)$ be the associated Eisenstein series.
Then $E_D(s,\Phi)$ has at most simple poles at $s = 1,3$ or $5$ and the corresponding residual representations are completely described in 
  \cite[Theorem 6.4]{HS}.  Set
  \[  \mathcal E = \{ {\rm Res}_{s=1} E_D(s,\Phi): \Phi \in I_D(s) \}, \]
We can now prove Theorem \ref{T:sw} in the introduction (which we restate here):
\vskip 5pt

\begin{thm}   \label{T:sw2}
Let $F$ be a totally real global field and $D$ a quaternion algebra over $F$. 
 Then we have the following identity in the space of automorphic representations $G_D(\mathbb A)$, 
\[ 
\mathcal E = \oplus _{i: D \rightarrow \mathbb O} \Theta(1) , 
\] 
where the sum is taken over all isomorphism classes of embeddings $i: D \rightarrow \mathbb O$ into octonion algebras over $F$. 
\end{thm} 
\vskip 10pt

\begin{proof}
 Comparing Proposition \ref{P:Theta1} with   \cite[Theorem 6.4]{HS}, one
 sees  that $\Theta(1)$ is isomorphic, as an abstract representation, to a summand of  $\mathcal E$. 
  In view of the multiplicity one result  in Proposition \ref{P:Theta1}(i),  it follows that $\Theta(1)$ is equal 
to that irreducible summand, as a subspace of the space of automorphic forms.
\vskip 5pt

Now recall that the dual pair $D^1\times G_D$ arises from an embedding of 
$D$ into an octonion algebra $\mathbb O$. Every such embedding is unique up to conjugacy by $\Aut(\mathbb O)$. However, given $D$ there are multiple 
octonion algebras over $F$ containing $D$.  An isomorphism class of octonion algebras $\mathbb O$  over $F$ is specified by the isomorphism class 
of its local completions $\mathbb O_v$ for real places $v$. At each real place,  we have two choices:  the classical octonion algebra and its split form. 
But $D_v$ embeds into both 
if and only if it is a quaternion division algebra. Hence the number of octonion algebras over $F$ containing $D$ is 
is $2^m$ where $m$ is the number of real places $v$ such that $D_v$ is the quaternion algebra. Now, by an easy check left to the reader, 
non-isomorphic $\mathbb O$ give non-isomorphic $\Theta(1)$. Moreover, using our description of $\Theta(1)$ in Proposition \ref{P:Theta1} and   \cite[Theorem 6.4]{HS}, 
one sees that  all those  possible $\Theta(1)$ sum up to $\mathcal E$. This proves the theorem.
\end{proof}

\vskip 10pt 

\subsection{Weak lifting} 
As explained in the introduction, we have the following see-saw dual pair in $G$: 

\[
 \xymatrix{
  \Aut(\mathbb O)  \ar@{-}[dr] \ar@{-}[d] & G_D
     \ar@{-}[d] \\
  D^1 \ar@{-}[ur] &  \PGSp_6}
\]

\vskip 10pt

\noindent Using this, we shall complete the proof of Theorem \ref{T:main} (which we reproduce here):
\vskip 5pt

 \begin{thm}  \label{T:main2}
  Suppose that $\pi$ is a cuspidal automorphic representation of $\PGSp_6$ (over $F=\mathbb Q$)  such that $L^S(s, \pi, {\rm Spin})$ has a pole at $s=1$. Then  there exists an octonion algebra $\mathbb O$ over $F$ and  a cuspidal automorphic representation $\pi'$ of $\Aut(\mathbb O)$ such that the Satake parameters of $\pi'$ 
  are mapped by $\iota$ to those of $\pi$ (i.e. $\pi$ is  a weak functorial lift of $\pi'$).   
  \vskip 5pt
  
  If the cuspidal representation $\pi$ of $\PGSp_6$ is tempered, then the following are equivalent:
  \vskip 5pt
  \begin{itemize}
  \item[(a)]  For almost all places $v$, the Satake parameter $s_v$ of $\pi_v$ is contained in $\iota(G_2(\mathbb C))$.
  
  \item[(b)] There exists an octonion algebra $\mathbb O$ over $F$ and  a cuspidal automorphic representation $\pi'$ of $\Aut(\mathbb O)$ such that  
  $\pi$ is a weak functorial lift of $\pi'$.
    \end{itemize}
      \end{thm}

\vskip 5pt

\begin{proof}
Let $\pi$ be an irreducible cuspidal automorphic representation of $\PGSp_6$ and consider its global theta lift $\pi'$ on $G_2$. It can be shown (by a standard computation of the constant term of the global theta lift) that $\pi;$ is contained in the space of cusp forms on $G_2$. This was explained  in \cite[Theorem 3.1]{GJ}, noting that the genericity assumption on $\pi'$ was not needed there. See also \cite[Proposition 5.2]{GG}.
\vskip 5pt

Now suppose that  the  partial (degree 8) spin $L$-function $L^S(s, \pi, {\rm Spin})$ of $\pi$ has a pole at $s=1$.  
In the first arXiv version of the paper \cite{P},  Theorem 9.4, A. Pollack 
has given a representation of $L^S(s, \pi, {\rm Spin})$ by an integral of an automorphic form $\phi\in\pi$ against the Eisenstein series $E_D(s, \Phi)$ for $\Phi\in I_D(s)$. 
In other words, there is an identity 
\[ 
 \int_{\PGSp_6(F) \backslash \PGSp_6(\A)} E_D(s,\Phi)  \phi (g)   \, dg  = I_S(\phi, \Phi, s)  L^S(s, \pi, {\rm Spin}) 
  \] 
 for some sufficiently large set of places $S$ of $F$, and where $I_S(\phi, \Phi, s)$ is a semi local integral over the places in $S$. (We note that Pollack has a slightly different 
 choice of the parameter of the Eisenstein series: his parameter $s'$ and our $s$ are related by $s=2s'-5$.) 
  Moreover, for any $s_0\in \mathbb C$, there exists $\phi\in \pi$ and $\Phi$ such that $I_S(\phi, \Phi, s_0)\neq 0$. Hence the assumption that 
${\rm Res}_{s=1} L^S(s,\pi, {\rm Spin})\neq 0$ implies that the integral of $\phi$ against some residue  
${\rm Res}_{s=1} E_D(s,\Phi)$ is non-zero.  Since the space of residues at $s=1$ is invariant under the complex conjugation,
 it follows that the integral of $\bar\phi$ against some residue  ${\rm Res}_{s=1} E_D(s,\Phi)$ is non-zero.
By the Siegel-Weil formula (Theorem \ref{T:sw2}), it follows that 
\[  
 \int_{\PGSp_6(F) \backslash \PGSp_6(\A)} \bar\phi(g) \cdot \left(  \int_{D^1(F) \backslash D^1(\A)}  \theta(f)(gh)  \, dh \right) \, dg  \ne 0 
 \] 
for some $\mathbb O \supset D$, $f \in \Pi_{\mathbb O}$ and $\phi \in \pi$, where $\theta(f)$ is rapidly decreasing on $D^1(F) \backslash D^1(\A)$ and of moderate 
growth on $\PGSp_6(\A)$. 
Exchanging the order of integration, we deduce that 
the global theta lift of $\pi$ to $\Aut(\mathbb O)$ is nonzero, i.e.
 \[ 
\phi'(h)= \int_{\PGSp_6(F)\backslash \PGSp_6(\mathbb A)} \theta(f)( g h) \bar\phi(g) ~dg 
\] 
is a non-zero function of uniform moderate growth on $\Aut(\mathbb O)\backslash \Aut(\mathbb O\otimes_F\mathbb A)$. 
It is given that $\phi$ is an eigenfunction for the center of the enveloping algebra of $\Aut(\mathbb O_v)$ for every real place $v$ of $F$. 
By \cite{HPS} and \cite{Li}, for every element $z'$ in the center of the enveloping algebra of $\PGSp_6(F_v)$, there exists an element in the 
center of the enveloping algebra of $\Aut(\mathbb O_v)$ such that $z=z'$, when acting on the minimal representation, in particular $z'\cdot f=z\cdot f$. 
Thus $\phi'$ is an eigenfunction for the enveloping algebra of $\PGSp_6(F_v)$ for every real place 
$v$ of $F$. (At this point we use that $\phi$ has rapid decrease to justify that differentiation of $f$ can be moved over to differentiation of $\phi$.) 
\vskip 5pt

Similarly,  it is given that $\phi$ is an eigenfunction for the Hecke algebra for almost all finite places. But so is $\phi'$ by matching of 
Hecke operators under the exceptional theta correspondences  \cite{SW15}. Moreover, by   \cite[Theorem 1.1]{SW15},  if $s'_v$ are the Satake conjugacy classes 
in $G_2(\mathbb C)$ corresponding to $\phi'$ and 
$s_v$ are the Satake conjugacy classes in $\Spin_7(\mathbb C)$ corresponding to $\phi$, 
 then $s_v=\iota(s'_v)$ where $\iota : G_2(\mathbb C) \rightarrow \Spin_7(\mathbb C)$ is the natural inclusion.  
 Hence, the submodule generated by all such global theta lifts $\phi'$ gives an automorphic representation $\pi'$ which weakly lifts to $\pi$.
 This proves the first assertion of the theorem.
 \vskip 10pt
 
 For the second part of the theorem,  it is clear that (b) implies (a). Conversely, 
as observed by Chenevier \cite[\S 6.12]{C}, the hypothesis (a) in the theorem implies that
 \[  L^S(s, \pi, {\rm Spin})  = \zeta^S(s) \cdot L^S(s, \pi, {\rm Std}) \]
 where the last L-function on the right is the degree 7 (partial) standard L-function of $\pi$. Since we are assuming that $\pi$ is tempered, it follows that 
$L^S(1, \pi, {\rm Std})$ is finite and nonzero. Hence $L^S(s, \pi, {\rm Spin})$ has a pole at $s=1$ and the results we have shown above imply that  (b) holds, with $\pi'$ the global theta lift of $\pi$ to $\Aut(\mathbb O)$. 

\vskip 5pt

 This completes the proof of the theorem.
 \end{proof}
 
 \vskip 5pt 
 We can strengthen our results in the case when $F=\mathbb Q$ and $\pi$ is a cuspidal representation of $\PGSp_6(\A)$ that corresponds to a classical Siegel 
 holomorphic form of positive weight. Recall that there are two isomorphism classes of octonion algebras over $\mathbb Q$: the classical octonion algebra 
 $\mathbb O^c$ and its split form $\mathbb O^s$. Then $\Aut(\mathbb O^c_{\infty})$ is an anisotropic group, while $\Aut(\mathbb O^s_{\infty})$ is split. 
 
 \begin{thm} Let $F=\mathbb Q$, and $\pi$ a cuspidal representation of $\PGSp_6(\A)$ that corresponds to a classical Siegel 
 holomorphic form  $\phi_{2r}$ of weight $2r>0$. If $L^S(s, \pi, {\rm Spin})$ has a pole at $s=1$, then $\pi$ is a lift from $\Aut(\mathbb O^c)$. Moreover, 
 if the level of $\phi_{2r}$ is one, then $\pi$ is a strong functorial lift from $\Aut(\mathbb O^c)$. 
  \end{thm} 
  \begin{proof} Let ${\mathrm U}_3(\mathbb R)$ be the maximal compact subgroup of $\Sp_6(\mathbb R)$. 
 By our assumption, $\pi_{\infty}$ is a lowest weight module, with the minimal ${\mathrm U}_3(\mathbb R)$-type $\det^{2r}$.  We need to show that 
 such representation does not occur in the exceptional theta correspondence with $\Aut(\mathbb O^s_{\infty})$. This correspondence is was studied in \cite{LS}. 
 It is shown there that if a representation of $\Sp_6(\mathbb R)$ that contains a type $\det^{2r}$ participates in the theta correspondence with $G_2(\mathbb R)$, it must be
 a lift of a spherical representation of $\Aut(\mathbb O^s_{\infty})$. 
 Truth be told, this was shown in \cite{LS} only for $r=0$, but the same argument works for any $r$. It was also shown in \cite{LS} that a spherical representation of $\Aut(\mathbb O^s_{\infty})$ 
  always lifts to a spherical representation of $\Sp_6(\mathbb R)$. But our $\pi_{\infty}$ is not spherical, hence it cannot appear in this correspondence. 
\vskip 5pt

 The correspondence for the dual pair  $\Aut(\mathbb O^c_{\infty})\times \Sp_6(\mathbb R)$ was completely determined in \cite{GrS} and is functorial. Thus, if $\phi_{2r}$ is of level 
 one, i.e. spherical at all primes, then $\pi$ is indeed a (strong) functorial lift from $\Aut(\mathbb O^c_{\infty})$.  
 
 \end{proof} 
\vskip 15pt

\noindent{\bf Acknowledgments:} W.T. Gan is partially supported by an MOE Tier 2 grant R146-000-233. G. Savin is partially supported by 
Simons Foundation grant 579347.

\end{document}